\documentclass[10pt]{amsart}

\usepackage{color,graphicx,shortvrb}
\usepackage[latin 1]{inputenc}

\usepackage{enumerate}
\usepackage{amssymb}
\usepackage{tikz-cd}

\usepackage [ all ]{xy}

\newtheorem{theorem}{Theorem}[section]

\newtheorem{corollary}[theorem]{Corollary}

\newtheorem{lemma}[theorem]{Lemma}

\newtheorem{proposition}[theorem]{Proposition}

\def\J#1#2#3{ \left\{ #1,#2,#3 \right\} }

\def\11{\textbf{$1$}}
\def\CC{{\mathbb{C}}}

\renewcommand{\Re}{\operatorname{Re}}

\DeclareGraphicsExtensions{.jpg,.pdf,.png,.eps}

\begin{document}

\title[Extension of isometries from the unit sphere]{Extension of isometries from the unit sphere of a rank-2 Cartan factor}

\author[O.F.K. Kalenda]{Ond\v{r}ej F.K. Kalenda}

\author[A.M. Peralta]{Antonio M. Peralta}

\address{Charles University, Faculty of Mathematics and Physics, Department of
Mathematical Analysis, Sokolovsk{\'a} 86, 186 75 Praha 8, Czech Republic}
\email{kalenda@karlin.mff.cuni.cz}
\address{Departamento de An{\'a}lisis Matem{\'a}tico, Facultad de
Ciencias, Universidad de Gra\-na\-da, 18071 Granada, Spain.}
\email{aperalta@ugr.es}




\date{}

\begin{abstract}
We prove that every surjective isometry from the unit sphere of a rank-2 Cartan factor $C$ onto the unit sphere of a real Banach space $Y$, admits an extension to a surjective real linear isometry from $C$ onto $Y$. The conclusion also covers the case in which $C$ is a spin factor. This result closes an open problem and, combined with the conclusion in a previous paper, allows us to establish that every JBW$^*$-triple $M$ satisfies the Mazur--Ulam property, that is, every surjective isometry from its unit sphere onto the unit sphere of a arbitrary real Banach space $Y$ admits an extension to a surjective real linear isometry from $M$ onto $Y$.
\end{abstract}

\maketitle
\thispagestyle{empty}

\section{Introduction}

\emph{Tingley's problem}, i.e. the question whether any surjective isometry between the unit spheres of two normed spaces admits a real-linear extension, has defined an active and fruitful line of research in recent years. This problem was named after D. Tingley who was the first author  studying this question in the setting of finite dimensional Banach spaces (see \cite{Ting1987} where he proved that such an isometry preserves antipodality). The reader should note that Tingley's problem remains open even for two-dimensional Banach spaces. The simplicity of the problem makes the question as attractive as difficult, and a fruitful mathematical machinery has been developed to find positive solutions to Tingley's problem in concrete classes of Banach spaces (see, for example, the surveys \cite{YangZhao2014,Pe2018} and the recent references \cite{CabSan19,FerJorPer2018,FerPe17c,FerPe17d,FerPe18Adv,Mori2017,PeTan16}).\smallskip

A Banach space $X$ satisfies the \emph{Mazur--Ulam property} if every surjective isometry from its unit sphere, $S(X)$, onto the unit sphere of any real Banach space admits an extension to a surjective real linear isometry between the corresponding spaces. This property was first termed by L. Cheng and Y. Dong in \cite{ChenDong2011}, probably due to the natural connections between Tingley's problem and the Mazur--Ulam theorem.  The study of the Mazur--Ulam property in different classes of Banach spaces is now a day a challenging subject of study for researchers (cf. \cite{CabSan19,CuePer18,CuePer19,JVMorPeRa2017,MoriOza2018,Pe2019,WH19}).\smallskip

For the sake of brevity, we shall focus on two recent contributions on the Mazur--Ulam property. In the first one, M. Mori and N. Ozawa prove that unital C$^*$-algebras and real von Neumann algebras are among the spaces satisfying the Mazur--Ulam property
(cf.\ \cite{MoriOza2018}). Additional examples of Banach spaces satisfying the Mazur--Ulam property have been found in \cite{BeCuFerPe2018}, where it is proved that if $M$ is a JBW$^*$-triple but not a Cartan factor of rank two, then $M$ satisfies the Mazur--Ulam property. The problem whether every rank-$2$ Cartan factor satisfies the Mazur--Ulam property remained as an intriguing open question. Among the examples of rank-2 Cartan factors which are not covered by the main result in \cite{BeCuFerPe2018} we find the spin factors which constitute an important model in physics (cf. \cite{BoHam2010,FriRu2001,JorvNeuWign34}). This note is aimed to present a complete solution to this problem. Our main result is the following theorem.

\begin{theorem}\label{t rank-2 Cartan factor of rank 2} Let $C$ be a rank-2 Cartan factor. Then $C$ satisfies the Mazur--Ulam property, that is, for each real Banach space $Y$, every surjective isometry $\Delta: S(C)\to S(Y)$ admits an extension to a surjective real linear isometry from $C$ onto $Y$.\end{theorem}

This result can be now combined with the main conclusion in \cite{BeCuFerPe2018} to deduce the following corollary.

\begin{corollary}\label{c rank-2 Cartan factor of rank 2} Every JBW$^*$-triple $M$ satisfies the Mazur--Ulam property, that is, every surjective isometry from the unit sphere of $M$ onto the unit sphere of an arbitrary real Banach space $Y$ can be extended to a surjective real linear isometry from $M$ onto $Y$.
\end{corollary}

Along this note, the closed unit ball of a Banach space $X$ will be denoted by $\mathcal{B}_X.$ The symbol $\mathbb{T}$ will stand for the unit sphere of $\mathbb{C}$. The basic notions and results on JB$^*$-triples and Cartan factors are surveyed in section \ref{sec:2}, where we also obtain some new results on the properties of the surjective isometries from the unit sphere of a JB$^*$-triple onto the unit sphere of a real Banach space.\smallskip

Let us briefly comment the strategy of the proof. The basic aim consists in verifying the assumptions of the following technical lemma.

\begin{lemma}\label{l 2.1 FangWang}{\rm\cite[Lemma 6]{MoriOza2018}, \cite[Lemma 2.1]{FangWang06})} Let  $\Delta : S(X) \to S(Y)$ be a surjective isometry between the unit spheres of two real normed spaces. Assume we can find two families of functionals $\{\varphi_i\}_i \subset \mathcal{B}_{X^*}$ and $\{\psi_i\}_i\subset \mathcal{B}_{Y^*}$ such that $\varphi_i = \psi_i \Delta$ for every $i$, and that the family $\{\varphi_i\}_i$ is norming for $X$. Then, $\Delta$ extends to a surjective real linear isometry.
\end{lemma}

Hence, Theorem~\ref{t rank-2 Cartan factor of rank 2} follows from the just mentioned lemma and Proposition~\ref{p behavior on pure atoms rank2 Cartan factor}, since the pure atoms addressed in this proposition are just extreme points of the unit ball of $C^*$, which form a norming set. To prove the final proposition we use the results from Section~\ref{sec:spin factors} on the structure of spin factors and on the behaviour of isometries on self-adjoint parts of Peirce-$2$ subspaces (see Proposition~\ref{p linearity on the hermitian part}) and some results on automorphisms of Cartan factors given in Section~\ref{sec: inner automorphisms}.

Let us further remark that some important steps of our proof are specific for the case of rank-$2$ Cartan factors, in particular one of the key steps consists in using Lemma~\ref{l distance 2 from a tripotent Cartan factor rank geq 2} precisely for rank-$2$ Cartan factors. So, the current paper is a real complement to the results of \cite{BeCuFerPe2018}, where the Mazur-Ulam property is proved for all JBW$^*$-triples except for Cartan factors of rank $2$. Let us also point out that many of the results of Section~\ref{sec:2} are proved in a more general setting, some of them even for general JB$^*$-triples. But investigation of Mazur-Ulam property for general non-dual JB$^*$-triples will probably need some new methods.

\section{JB$^*$-triples and rank}\label{sec:2}

A \emph{JB$^*$-triple}, as introduced in \cite{Ka83}, is a complex Banach space $E$ admitting a continuous triple product $\J \cdot\cdot\cdot :
E\times E\times E \to E,$ which is symmetric and bilinear in the first and third variables, conjugate linear in the second variable,
and satisfies the following axioms:
\begin{enumerate}[{\rm (a)}] \item (Jordan identity)
$$L(a,b) L(x,y) = L(x,y) L(a,b) + L(L(a,b)x,y)
 - L(x,L(b,a)y)$$ for $a,b,x,y$ in $E$, where $L(a,b)$ is the operator on $E$ given by $x \mapsto \J abx;$
\item $L(a,a)$ is a hermitian operator with non-negative spectrum for all $a\in E$;
\item $\|\{a,a,a\}\| = \|a\|^3$ for each $a\in E$.\end{enumerate}

The class of JB$^*$-triples contains, but is not limited to, all C$^*$-algebras and the spaces $B(H,K)$, of all bounded linear operators between complex Hilbert spaces $H$ and $K$, with triple product \begin{equation}\label{eq C*-triple product} \J xyz = \frac12 (x y^* z + z y^* x).
 \end{equation} It follows than any complex Hilbert space is a JB$^*$ triple (when identified with $B(\CC,H)$).

An important subclass of JB$^*$-triples is formed by JB$^*$-algebras. Recall that a  real (respectively, complex) \emph{Jordan algebra} is a
(not-necessarily associative) algebra over the real (respectively, complex) field whose product is abelian and satisfies the Jordan identity: $$(a \circ
b)\circ a^2 = a\circ (b \circ a^2).$$ A \emph{Jordan Banach algebra} is a normed Jordan algebra $A$ whose norm is complete and satisfies $\|
a\circ b\| \leq \|a\| \ \|b\|$, $a,b\in A$. A \emph{JB$^*$-algebra} is a complex Jordan Banach algebra $A$ equipped with an algebra involution $^*$ satisfying  $$\|\J a{a^*}a \|= \|a\|^3, \hbox{ for all $a\in A$},$$ where $\J a{a^*}a  = 2 (a\circ a^*) \circ a - a^2 \circ a^*$. If $A$ is a C$^*$-algebra, it becomes a JB$^*$-algebra when equipped with the Jordan product $a\circ b=\frac12(ab+ba)$. Moreover,
by \cite[Theorem 3.3]{BraKaUp78}, every JB$^*$-algebra $A$ becomes a JB$^*$-triple when equipped with the triple product
$$\J abc = (a \circ b^*) \circ c + (c\circ b^*) \circ a - (a\circ c) \circ b^*.$$ 

By analogy with von Neumann algebras, a \emph{JBW$^*$-triple} is a JB$^*$-triple which is also a dual Banach space (and, similarly, a \emph{JBW$^*$-algebra} is a JB$^*$-algebra which is also a dual Banach space). The bidual of a JB$^*$-triple is a JBW$^*$-triple with respect to a triple product extending the one of $E$ \cite{Di86}. J.T. Barton and R.M. Timoney proved in \cite{BarTi86} that every JBW$^*$-triple admits a unique isometric predual and its triple product is separately weak$^*$ continuous.\smallskip

Additional examples of JB$^*$- and JBW$^*$-triples are given by the so-called \emph{Cartan factors}. Suppose $H_1$ and $H_2$ are two complex Hilbert spaces, the triple product given in \eqref{eq C*-triple product} defines an structure of JB$^*$-triple on the space $L(H_1,H_2)$, of all bounded linear operators between $H_1$ and $H_2$. Those JB$^*$-triples of the form $L(H_1,H_2)$ are called \emph{Cartan factors of type 1}. Clearly the space $K(H_1,H_2)$, of all compact operators from $H_1$ into $H_2$ is a JB$^*$-subtriple of $L(H_1,H_2)$. In order to describe the next two Cartan factors, let $j$ be a conjugation (i.e. a conjugate linear isometry or period 2) on a complex Hilbert space $H$. The assignment $x\mapsto x^t:=jx^*j$ defines a linear involution on $L(H)$ (which can be represented as the transpose with respect to a suitable orthonormal basis). A \emph{Cartan factor of type 2} (respectively, of \emph{type 3}) is a complex Banach space which coincides with the JB$^*$-subtriple of $L(H)$ of all $t$-skew-symmetric (respectively, $t$-symmetric) operators.\smallskip

The \emph{Cartan factors of type 4}, also called \emph{spin factors}, are defined as complex Hilbert spaces $X$ provided with a conjugation $ x\mapsto \overline{x},$ with the triple product
and the norm  given by \begin{equation}\label{eq spin product}
\{x, y, z\} = \langle x|y\rangle z + \langle z|y\rangle  x -\langle x|\overline{z}\rangle \overline{y},
\end{equation} and \begin{equation}\label{eq spin norm} \|x\|^2 = \langle x|x\rangle  + \sqrt{\langle x|x\rangle ^2 -|
\langle x|\overline{x}\rangle  |^2},
 \end{equation} respectively. All we need to know about Cartan factors of types 5 and 6 is that they are finite dimensional (see \cite[\S 3]{Ka81} or \cite[page 199]{Ka97}) for additional details).\smallskip

Let $E$ be a JB$^*$-triple. Elements in $E$ which are fixed points for the triple product are called \emph{tripotents}. Each tripotent $e\in {E}$ induces a decomposition of ${E}$ in terms of the eigenspaces of the operator $L(e,e)$ given by
\begin{equation}\label{Peirce decomp} {E} = {E}_{0} (e) \oplus  {E}_{1} (e) \oplus {E}_{2} (e),\end{equation} where ${
E}_{k} (e) := \{ x\in {E} : L(e,e)x = {\frac k 2} x\}$ is a subtriple of ${E}$ (for $k=0,1,2$). The natural projection of ${E}$ onto ${E}_{k} (e)$ is called the Peirce-$k$ projection and will be denoted by $P_{k} (e)$. We shall apply later that Peirce projections are all contractive (cf. \cite{FriRu85}). The so-called \emph{Peirce rules} predict the triple products among Peirce subspaces in the following way:
$$\J {{E}_{k}(e)}{{E}_{l}(e)}{{E}_{m}(e)}\! \subseteq {E}_{k-l+m} (e),\!\!\hbox{ and }\!\! \J {{E}_{0}(e)}{{E}_{2} (e)}{{E}}\! =\! \J {{E}_{2} (e)}{{E}_{0} (e)}{{E}}\! =\! \{0\},$$
where ${E}_{k-l+m} (e) = \{0\}$ whenever $k-l+m$ is not in $\{0,1,2\}$. Another connection with the Jordan theory tells that ${E}_{2} (e)$ is a unital
JB$^*$-algebra with respect to the product and involution given by $x \circ_e y = \J xey$ and $x^{*_e} = \J exe,$ respectively. Furthermore,  $E_{2} (e)$ is a JBW$^*$-algebra when $E$ is a JBW$^*$-triple. The self-adjoint or hermitian part of $E_2(e)$ will be denoted by $E_2(e)_{sa}$, that is, $$E_2(e)_{sa}=\{x\in E_2(e) : x^{*_e} = x \}.$$

A tripotent $e$ in $E$ is called \emph{minimal} (respectively, \emph{complete} or \emph{maximal}) if  $E_2(e)=\CC e \neq \{0\}$ (respectively, $E_0 (e) =\{0\}$). We shall say that $e$ is a \emph{unitary tripotent} if $E_2(e) =E$.\smallskip

Two tripotents $u,v$ in a JB$^*$-triple $E$ are called \emph{collinear} ($u\top v$ in short) if $u\in E_1(v)$ and $v\in E_1(u)$. We shall say that $u$ \emph{governs} $v$  ($u \vdash v$ in short) if $v\in U_{2} (u)$ and $u\in U_{1} (v)$.\smallskip

Elements $x$ and $y$ in a JB$^*$-triple $E$ are called \emph{orthogonal} ($x\perp y$ in short) if $L(x,y)=0$ (equivalently $L(y,x)=0$, compare \cite[Lemma 1.1]{BurFerGarMarPe}). If $e$ and $v$ are tripotents in $E$, it can be shown that $e \perp v$ if and only if $e\in E_{0} (v)$. A subset $\mathcal{S}\subseteq E$ is said to be \emph{orthogonal} if $0\notin \mathcal{S}$ and $x\perp y$ for every $x\neq y$ in $\mathcal{S}$. The \emph{rank} of a JB$^*$-triple $E$ is defined as the minimal cardinal number $r$ satisfying $\hbox{card}(\mathcal{S})\leq r$ whenever $\mathcal{S}$ is an orthogonal subset of $E$. It is known that a JB$^*$-triple $E$ is a reflexive Banach space if it has finite rank (cf. \cite[Proposition 4.5]{BuChu} and \cite[Theorem 6]{ChuIo90}).\smallskip

We shall also employ the natural partial order on the set of tripotents in a JB$^*$-triple $E$ defined in the following way: Given two tripotents $u,e$ in $E$, we shall say that $u \leq e$ if $e-u$ is a tripotent in $E$ and $e-u \perp u$. \smallskip

Let $\varphi$ be a norm-one functional in the dual, $E^*$, of a JB$^*$-triple $E$. Suppose $e$ is a tripotent in $E$ satisfying $\varphi(e) = 1$. Then $\varphi = \varphi P_2(e)$ and $\varphi|_{E_2(e)}$ is a positive norm-one functional in the dual of the JB$^*$-algebra $E_2(e)$ (cf. \cite[Proposition 1]{FriRu85}). Another technical result, due to Friedman and Russo, required for later purposes, affirms the following: \begin{equation}\label{eq FR 1.6} e\in E\mbox{ is a tripotent}, \ x\in S(E) \hbox{ with } P_2(e)(x) = e \Rightarrow x = e+ P_0(e) (x),
\end{equation} (see \cite[Lemma 1.6]{FriRu85}). We shall say that a tripotent $e\in E$ has rank $k$ if the JB$^*$-triple $E_2(e)$ has the same rank.\smallskip

Let $X$ be a real or complex Banach space with dual space $X^*$. Suppose $F$ and $G$ are two subsets of $\mathcal{B}_{X}$ and $\mathcal{B}_{X^*}$, respectively.  Then we set
$$ F^{\prime} =F^{\prime,X^*} =  \{a \in \mathcal{B}_{X^*}:a(x) = 1 \mbox{ for } x \in F\},$$
$$G_{\prime} =G_{\prime,X}= \{x \in \mathcal{B}_{X}:a(x) = 1\mbox{ for } a \in G\}.$$
Clearly, $F^{\prime}$ is a weak$^*$-closed face of $\mathcal{B}_{X^*}$ and $G_{\prime}$
is a norm closed face of $\mathcal{B}_{X}$. We say that $F$ is a \emph{norm-semi-exposed face} of $\mathcal{B}_{X}$ (respectively, $G$ is a \emph{weak$^*$-semi-exposed face} of $\mathcal{B}_{X^*}$) if $F=(F^{\prime})_{\prime}$ (respectively, $G= (G_{\prime})^{\prime}$). It is known that the mappings $F \mapsto F^{\prime}$
and $G \mapsto G_{\prime}$ are anti-order isomorphisms between the
complete lattices $\mathcal{S}_n(\mathcal{B}_{X})$, of norm-semi-exposed faces
of $\mathcal{B}_X,$ and $\mathcal{S}_{w^*}( \mathcal{B}_{X^*}),$ of weak$^*$-semi-exposed
faces of $\mathcal{B}_{X^*}$, and are inverses of each other. A face of $\mathcal{B}_X$ is called proper if it does not coincide with the whole $\mathcal{B}_X$, therefore every proper face of $\mathcal{B}_X$ is contained in the unit sphere of $X$.\smallskip

In a JBW$^*$-triple $M$, every weak$^*$-closed face of $\mathcal{B}_{M}$ is weak$^*$-semi-exposed and the assignment \begin{equation}\label{eq anti order weak* closed faces complex} u \mapsto
(\{u\}_{\prime})^{\prime} = u + \mathcal{B}_{M_0(u)}
\end{equation} is an anti-order isomorphism from the partially ordered set of tripotents in $M$ onto the partial ordered of weak$^*$-closed faces of $\mathcal{B}_{M}$ (cf. \cite[Theorem 4.6]{EdRutt88}). Furthermore, every norm closed face of $\mathcal{B}_{M_{*}}$ is norm-semi-exposed and is of the form  $\{u\}_{\prime}$ for a tripotent $u\in M$ (see \cite[Theorem 4.4]{EdRutt88}). The main results in \cite{EdFerHosPe2010} and \cite{FerPe10} establish that norm closed faces of the closed unit ball of a JB$^*$-triple $E$ and weak$^*$ closed faces of the closed unit ball of its dual are in one-to-one correspondence with the compact tripotents in $E^{**}$.\smallskip

Let us give the first observation on the metric structure of the unit sphere of a JB$^*$-triple.

\begin{lemma}\label{l distance 2 from a tripotent} Let $e$ be a non-zero tripotent in a JB$^*$-triple $E$. Suppose $x$ is a norm-one element in $E$ such that $\|e+ x\| =2$ then $\|e+ P_2(e)(x)\| =2$ and $\|P_2(e)(x)\| =1$.
\end{lemma}

\begin{proof} Let us take a norm-one functional $\varphi\in E^*$ satisfying $\varphi (e + x) =2$. It then follows that $1=\varphi(e) = \varphi (x) = \varphi P_2(e) (x) \leq \|P_2(e)(x)\|\leq 1,$ and thus $2= \varphi (e+P_2(e)(x)) \leq \|e+P_2(e)(x)\| \leq  2,$ which gives the desired statement.
\end{proof}

One of the interesting geometric properties of JB$^*$-triples guarantees that the extreme points of the closed unit ball of a JB$^*$-triple $E$ are precisely the complete tripotents in $E$
(cf. \cite[Lemma 4.1]{BraKaUp78} and \cite[Proposition 3.5]{KaUp77}).\smallskip

Every proper norm closed face of the closed unit ball of a JB$^*$-triple $E$ is norm-semi-exposed (cf. \cite[Corollary 3.11]{EdFerHosPe2010}). It is shown in the proof of \cite[Proposition 2.4]{FerGarPeVill17} that every norm-semi-exposed face of the closed unit ball of $E$ is an intersection face in the sense employed in \cite[Lemma 8]{MoriOza2018}, that is, every norm-semi-exposed face of the closed unit ball of $E$ coincides with the intersection of all maximal proper norm closed faces containing it. Combining these arguments with the just quoted Lemma 8 in \cite{MoriOza2018} we get the following lemma.

\begin{lemma}\label{l intersection faces MoriOzawa L8}{\rm(\cite[Lemma 8]{MoriOza2018}, \cite[Proposition 2.4]{FerGarPeVill17}, \cite[Corollary 3.11]{EdFerHosPe2010})} Let $\Delta: S(E)\to S(Y)$ be a surjective isometry where $E$ is a JB$^*$-triple and $Y$ is a real Banach space. Then $\Delta$ maps proper norm closed faces of $\mathcal{B}_{E}$ to intersection faces in $S(Y)$. Furthermore, if $F$ is a proper norm closed face of $\mathcal{B}_{E}$ then $\Delta(-F) = -\Delta(F)$.
\end{lemma}

The next corollary is a straightforward consequence of the previous lemma.

\begin{corollary}\label{c antipodal for extreme points} Let $\Delta: S(E)\to S(Y)$ be a surjective isometry where $E$ is a JB$^*$-triple and $Y$ is a real Banach space. Suppose $u$ is a complete tripotent in $E$, then $\Delta(u)$ is an extreme point of $\mathcal{B}_Y$ and $\Delta(-u) = -\Delta(u)$.
\end{corollary}

\begin{corollary}\label{c Tingley antipodal thm for tripotents} Let $\Delta: S(E)\to S(Y)$ be a surjective isometry where $E$ is a JB$^*$-triple and $Y$ is a real Banach space. Suppose $e$ is a non-zero tripotent in $E$, then $\Delta(-e) = -\Delta(e)$.
\end{corollary}

\begin{proof} Let us consider the norm closed proper face $F_e^E = e+\mathcal{B}_{E_0(e)}$. It is easy to check that $e$ (respectively, $-e$) is the unique element in $F_e^E$ (respectively, in $-F_e^E$) whose distance to any other element in $F_e^E$  (respectively, in  $-F_e^E$) is smaller than or equal to $1$. Therefore $\Delta(-e)$ (respectively, $\Delta(e)$) is the unique element in $\Delta(-F_e^E)$ (respectively, in $\Delta(F_e^E)$) satisfying $\|\Delta(-e) - b\|\leq 1$ for all $b\in \Delta(-F_e^E)$ (respectively, in $\Delta(F_e^E)$). Lemma \ref{l intersection faces MoriOzawa L8} assures that $-\Delta(e)\in \Delta(-F_e^E),$ and for each $b\in \Delta(-F_e^E) = - \Delta(F_e^E),$ there exists $x\in F_v^E$ such that $-b= \Delta(x)$, then we have $$\|-\Delta(e) - b\| = \|-\Delta(e) + \Delta(x) \| = \|x-e \|\leq 1,$$ which guarantees that $\Delta(-e) = -\Delta(e)$.
\end{proof}

It follows from the study on the geometric structure of the predual of a JBW$^*$-triple in \cite{FriRu85} that the extreme points in the closed unit ball of the dual space, $E^*$, of a JB$^*$-triple $E$ are in one-to-one correspondence with the minimal tripotents in $E^{**}$ via the following correspondence:
\begin{equation}\label{eq pure atoms and minimal partial isometries}\begin{gathered} \hbox{For each $\varphi\in\mbox{ext\,} \mathcal{B}_{E^*}$ there exists a unique minimal tripotent $v\in E^{**}$}\\ \hbox{satisfying $\varphi(x) v = P_2(v) (x)$ for all $x\in E^{**}$},
\end{gathered}\end{equation} (see \cite[Proposition 4]{FriRu85}). Extreme points of $\mathcal{B}_{E^*}$ are called \emph{pure atoms}. For each minimal tripotent $v$ in $E^{**}$, we shall write $\varphi_v$ for the unique pure atom associated with $v$.\smallskip

Another ingredient in our arguments is related to the facial structure of JB$^*$-triples. By the JB$^*$-triple version of Kadison's transitivity theorem (see \cite[Theorem 3.3]{BuFerMarPe}), each maximal norm closed proper face of $\mathcal{B}_E$ is of the form \begin{equation}\label{eq maximal proper faces} F_v^{E}= (v + \mathcal{B}_{_{E_0^{**}(v)}}) \cap E,
\end{equation} where $v$ is a minimal tripotent in $E^{**}$ (see \cite[Corollary 3.5]{BuFerMarPe} and \cite{EdFerHosPe2010}).\smallskip

The following two results have been borrowed from \cite{BeCuFerPe2018}.

\begin{lemma}\label{l existence of support functionals for the image of a face}\cite[Lemma 4.7]{BeCuFerPe2018} Let $E$ be a JB$^*$-triple and let $Y$ be a real Banach space. Suppose $\Delta : S(E)\to S(Y)$ is a surjective isometry. Then for each maximal proper norm closed face $F$ of the closed unit ball of $E$ the set $$\hbox{supp}_{\Delta}(F) := \{\psi\in Y^* : \|\psi\|=1,\hbox{ and } \psi^{-1} (\{1\})\cap \mathcal{B}_{Y} = \Delta(F) \}$$ is a non-empty weak$^*$ closed face of $\mathcal{B}_{Y^*}$; in other words, for each minimal tripotent $v$ in $E^{**}$ the set $$ \hbox{supp}_{\Delta}(F_v^{E}) := \{\psi\in Y^* : \|\psi\|=1,\hbox{ and } \psi^{-1} (\{1\})\cap \mathcal{B}_{Y} = \Delta(F_v^E) \}$$ is a non-empty weak$^*$ closed face of $\mathcal{B}_{Y^*}.$
\end{lemma}

The next corollary is one of the consequences of the fact that the closed unit ball of every JBW$^*$-triple $M$ satisfies the \emph{strong
Mankiewicz property} (cf. \cite[Corollary 2.2]{BeCuFerPe2018}).

\begin{corollary}\label{c closed faces associated with tripotents down}\cite[Corollary 4.1]{BeCuFerPe2018} Let $M$ be a JBW$^*$-triple, let $Y$ be a Banach space, and let $\Delta : S(M)\to S(Y)$ be a surjective isometry. Suppose $e$ is a non-zero tripotent in $M$, and let $F^M_e = e+ \mathcal{B}_{M_0(e)}=\left(e+ \mathcal{B}_{M^{**}_0(e)}\right)\cap M$ denote the proper norm closed face of $\mathcal{B}_{M}$ associated with $e$. Then the restriction of $\Delta$ to $F^{M}_e$ is an affine function. Furthermore, there exists a real linear isometry $T_e$ from $M_0(e)$ onto a norm closed subspace of $Y$ satisfying $\Delta( {e}+ x ) = T_e (x) +\Delta(e)$ for all $x\in \mathcal{B}_{M_0(e)}.$
\end{corollary}

If $E$ is a reflexive JB$^*$-triple, then all minimal tripotents in $E^{**}$ are actually in $E$, and hence all maximal proper norm closed faces of $\mathcal{B}_E$ are of the form $F_e^E$, where $e$ is a minimal tripotent in $E$. Since, by Zorn's lemma every convex subset in $S(E)$ is contained in a maximal convex subset, the next result is a consequence of the previous Corollary \ref{c closed faces associated with tripotents down} and these comments (compare also \cite[Lemma 3.2]{Tan2016preprint}).

\begin{corollary}\label{c affine on convex subsets of the sphere in reflexive}  Suppose $\Delta:S(E)\to S(Y)$ is a surjective isometry,  where $Y$ is a real Banach space and $E$ is a reflexive JB$^*$-triple.
  Then $\Delta|_{F}$ is affine for each convex subset $F\subset S(E)$.
\end{corollary}

Our next corollary gathers some interesting consequences. We recall that every JB$^*$-triple having finite rank is reflexive and hence a JBW$^*$-triple (cf. \cite[Proposition 4.5]{BuChu}).

\begin{corollary}\label{c additivity on mutually orthogonal minimal tripotents} Let $\Delta: S(E) \to S(Y)$ be a surjective isometry, where $Y$ is a real Banach space and $E$ is a finite rank JB$^*$-triple with rank at least two. Suppose $e_1$ is a minimal tripotents in $E$. Let $T_{e_1}: E_0(e_1)\to Y$ be the real linear isometry satisfying $\Delta( {e_1}+ x ) = \Delta(e_{1}) + T_{e_1} (x)$ for all $x\in \mathcal{B}_{E_0(e_1)},$ whose existence is given by Corollary \ref{c closed faces associated with tripotents down}. Then the following assertions hold:\begin{enumerate}[$(a)$]\item $\Delta(e_1+e_2) = \Delta(e_1) + \Delta(e_2),$ for every tripotent $e_2\in S(E_0(e_1))$;
\item $T_{e_1} (e_2) = \Delta (e_2),$ for every tripotent $e_2\in S(E_0(e_1))$;
\item Let $\widetilde{\Delta}: S(E) \to S(Y)$ be another surjective isometry. Then $\Delta = \widetilde{\Delta}$ if and only if $\Delta (v) = \widetilde{\Delta} (v)$ for every minimal tripotent $v$ in $E$;
\item Let $\varphi\in E^*$ and $\psi\in Y^*$ be functionals. Then the following statements are equivalent:
\begin{enumerate}[$(d.1)$]\item $\psi\Delta (z) = \Re\varphi(z),$ for all $z\in S(E)$;
\item $\psi\Delta (e) = \Re \varphi(e),$ for every complete tripotent $e\in E$;
\item $\psi\Delta (v) = \Re \varphi(v),$ for every minimal tripotent $v\in E$.
\end{enumerate}
\end{enumerate}
\end{corollary}

\begin{proof}$(a)$  By Corollary \ref{c closed faces associated with tripotents down} $\Delta|_{F_{e_j}^E}$ is an affine function for every $j=1,2$. So, by Corollary \ref{c antipodal for extreme points} we have
$$ \Delta(e_1)=  \Delta\left( \frac12 (e_1+e_2) + \frac12 (e_1-e_2)\right) = \frac12 \Delta(e_1+e_2) + \frac12 \Delta(e_1-e_2),$$ and
$$\begin{aligned}[t]
\Delta(e_2) &=  \Delta\left( \frac12 (e_1+e_2) + \frac12 (-e_1+e_2)\right)= \frac12 \Delta(e_1+e_2) + \frac12 \Delta(-e_1+e_2)\\
&= \frac12 \Delta(e_1+e_2) - \frac12 \Delta(e_1-e_2),
\end{aligned}$$ where in the last equality we applied Corollary \ref{c Tingley antipodal thm for tripotents}. Both identities together give $\Delta(e_1) + \Delta(e_2) = \Delta(e_1+e_2)$ as desired.\smallskip

$(b)$ The element $e_1 + e_2$ lies in $F_{e_1}^E$. We deduce from Corollary \ref{c closed faces associated with tripotents down} and $(a)$ that $$\Delta(e_1) + T_{e_1} (e_2) = \Delta(e_1+ e_2) = \Delta(e_1) + \Delta(e_2),$$ and hence $T_{e_1} (e_2) = \Delta(e_2)$.\smallskip

$(c)$ The ``only if'' implication is clear. Let us assume that $\Delta (v) = \widetilde{\Delta} (v)$ for every minimal tripotent $v$ in $E$. By applying that $E$ has finite rank we deduce that any complete tripotent in $E$ is the sum of a finite collection of pairwise orthogonal minimal tripotents. Hence by $(a)$ we get that $\Delta (v) = \widetilde{\Delta} (v)$ for every complete tripotent $v$ in $E$. It follows that $\Delta$ and $\widetilde{\Delta}$ coincide on each proper closed face. Indeed, let $F\subset S(E)$ be a proper closed face. Then $\Delta$ and $\widetilde{\Delta}$ are two continuous affine mappings on $F$ (by Corollary~\ref{c affine on convex subsets of the sphere in reflexive}) which coincide on extreme points of $F$ (note that extreme points of $F$ are also extreme points of $\mathcal{B}_E$, thus complete tripotents), so they coincide on $F$ by the  Krein-Milman theorem (note that $F$ is weakly compact due to the reflexivity of $E$).
Since any element of the sphere is contained in a proper closed face by the Hahn-Banach theorem, we see that $\Delta$ and $\widetilde{\Delta}$ coincide.\smallskip

$(d)$ The implication $(d.1) \Rightarrow (d.3)$ is clear. The implication $(d.3)\Rightarrow(d.2)$ follows from $(a)$ using the fact that any complete tripotent is  the sum of a  finite collection of pairwise orthogonal minimal tripotents.\smallskip

$(d.2)\Rightarrow(d.1)$ We proceed similarly as in $(c)$. The right-hand side is a continuous affine mapping on $E$, the left-hand side is a continuous mapping which is affine at each proper closed face of $\mathcal{B}_E$. Thus, having the equality on extreme points, we get equality at each proper closed face, hence on $S(E)$.

\end{proof}

\section{Structure of spin factors and applications to finite-rank Cartan factors}\label{sec:spin factors}

We recall the following result from \cite{MoriOza2018}.

\begin{lemma}\label{l MO l21}\cite[Lemma 21]{MoriOza2018} Let $1$ denote the unit element of the C$^*$-algebra $A=M_2(\mathbb{C})$ and let $\operatorname{tr}$ denote the normalized trace. The real linear subspace $$\mathcal{H} :=\{ x\in A_{sa} : \operatorname{tr} (x) = 0\}$$ is a real Hilbert space with $A_{sa} = \mathbb{R} 1\oplus^{1} \mathcal{H}$, and if an element $x\in S(A)$ satisfies $\| 1\pm x\|=2,$ then $x\in \mathcal{H}$.
\end{lemma}

In this section we shall establish a similar conclusion for spin factors and then apply it for general Cartan factors.\smallskip

Along this section $X$ will stand for a fixed spin factor.

Set
$$X^{-} =\{x\in X : \overline{x} =x\}.$$
Then $\|x\|= \|x\|_2 = \sqrt{\langle x|x\rangle}$ for all $x\in X^{-}$, hence $X^-$ is a real Hilbert space.
In particular, $\langle a | b\rangle = \langle b| a\rangle \in \mathbb{R}$ for all $a,b\in X^{-}$.
Further, clearly $X= X^{-} \oplus i X^{-}$, and hence dim$_{\mathbb{C}}(X)=$dim$_{\mathbb{R}}(X^{-})$. \smallskip

It is easy to check that every norm-one element $e\in X^{-}$ is a unitary element in $X$, i.e., it is a tripotent with $X_2(e) = X$. Moreover, given another norm-one element $y\in X^{-}$ with $\langle e| y\rangle =0,$ the elements $e_1= \frac{e+iy}{2}$ and $e_2= \frac{e-iy}{2}$ are two mutually orthogonal minimal tripotents in $X$ with $e = e_1+e_2$. It can be also checked that  $$X_1(e_1) =X_1(e_2)= \{{c+id}: c,d\in X^{-}, \ e,y \perp_2 c,d \}=\{e,y\}^{\perp_2},$$ where  $\perp_2$ is used to denote  orthogonality in the Hilbert space $(X,\langle\cdot|\cdot\rangle)$.
\smallskip

If $X$ is one dimensional, then $X= \mathbb{C}$ has rank one. When $X$ is two-dimensional we find an orthonormal basis $\{x_1,x_2\}$ of $X^-$. The elements $e_1= \frac{x_1+ix_2}{2}$ and $e_2= \frac{x_1-ix_2}{2}$ are mutually orthogonal minimal tripotents in $X$, therefore $X \cong \mathbb{C} e_1 \oplus^{\infty} \mathbb{C} e_2$ is not a factor. For these reasons it is standard to assume that dim$(X)\geq 3$ and we shall do so. \smallskip

It is known that every spin factor (unless the one-dimensional case) has rank $2$ (this easily follows from the above description of tripotents, see also \cite[Table 1 in page 210]{Ka97} or \cite[Table in page 475]{Ka81}).  
\smallskip

Let us recall that any two minimal tripotents in a Cartan factor $C$ are interchanged by a triple automorphism on $C$ (cf. \cite[Proposition 5.8]{Ka97}). The case of spin factors is easy and can be described in a canonical way. It is done in the following lemma which easily follows from the above description of tripotents.\smallskip

\begin{lemma}\label{L:shift in spin}
\begin{enumerate}[$(i)$]
	\item Let $T\in L(X^-)$ be a unitary operator on the real Hilbert space $X^-$ and let $\alpha$ be a complex unit. Then the operator
	$$a+ib\mapsto \alpha (T(a)+iT(b)),\quad a,b\in X^-,$$
	is a triple automophism of the spin factor $X$ which is simultaneously a unitary operator on the Hilbert space $(X,\langle\cdot|\cdot\rangle)$. Actually, every triple automorphism on $X$ is of this form;
	\item Any two unitary tripotents in $X$ are interchanged by a triple automorphism of the form from $(i)$;
	\item Any two minimal tripotents in $X$ are interchanged by a triple automorphism of the form from $(i)$ {\rm(}we can obtain even $\alpha=1${\rm)}.
\end{enumerate}
\end{lemma}

\begin{proof} Statement $(i)$ is in \cite[Theorem in page 196]{HervIs92}. The other statements are consequences of the first one.
\end{proof}

We continue by an extension of the first statement of Lemma~\ref{l MO l21} to spin factors.

\begin{proposition}\label{p hermitian part of the Peirce-2 subspace of a rank-2 trip in a spin} Let $e$ be a unitary (i.e. rank-2) tripotent in a spin factor $X$. Then
$$X_2(e)_{sa} = \mathbb{R} e \oplus^{1} \mathcal{H}_e,$$
where
$$\mathcal{H}_e=\{e\}^{\perp_2}_{X_2(e)_{sa}}$$
 is a real Hilbert space contained in $X_2(e)_{sa}$.
  Furthermore, on $\mathcal{H}_e$ the norms $\|\cdot\|$ and $\|\cdot\|_2$ coincide.
\end{proposition}

\begin{proof} By Lemma~\ref{L:shift in spin}$(ii)$ it is enough to prove the statement for one suitably chosen unitary tripotent. So, take the tripotent $ie$ where $e\in X^-$ is a norm-one element. Let us describe  the hermitian part of $X_2(i e)$.
Suppose $x\in X_2(i e)_{sa}$, that is $$x=\{i e, x, i e\} = - 2\langle e| x\rangle e + \overline{x},$$ and consequently, $$- \langle e| x\rangle e= \frac{x-\overline{x}}{2}= - \overline{\frac{x-\overline{x}}{2}} =  \overline{\langle e| x\rangle} e.$$ Therefore $\langle e| x\rangle= i t$ with $t\in \mathbb{R}$ and $x = a -i t e$ for some $a\in X^-$. Since $\langle e | x\rangle = \langle e| a\rangle +i t$, it follows from the above that $$a -i t e = x= - 2\langle e| x\rangle e + \overline{x}= -2 \langle e|a\rangle e - 2 i t e +a+ t i e,$$ which proves that $\langle e|a\rangle =0$. Conversely, if $a\in X^-$ and $\langle e|a\rangle=0$, then clearly $a-ite\in X_2(ie)_{sa}$. Thus
$$X_2(i e)_{sa} = \mathbb{R} i e \oplus \mathcal{H}_{ie},$$
where
$$\mathcal{H}_{ie}=\{e\}^{\perp_2}_{X^-}=\{ie\}^{\perp_2}_{X_2(i e)_{sa}}$$
is a real Hilbert space and $\|z\|= \|z\|_2$ for all $z\in \mathcal{H}_{ie}$. It remains to compute the norm. Accordingly to the formula of the spin norm in \eqref{eq spin norm}, for each $a+i t e$ with $t\in \mathbb{R}$, $a\in \{e\}^{\perp_2}_{X^-}$, we have $$\|a+i t e\|^2 = \|a+i t e\|_2^2 +\sqrt{\|a+i t e\|_2^4 - |\langle a+i t e| \overline{a+i t e}\rangle |^2} = \left( \|a\|_2 + |t|\right)^2.$$ We have shown that \begin{equation}\label{eq self-adjoint part Peirce2 unitary spin} X_2(i e)_{sa} =  i \mathbb{R} e\oplus^{1}\{e\}^{\perp_2}_{X^-}
\end{equation}
and the proof is completed. Observe that, unlike in Lemma~\ref{l MO l21}, the space $\mathcal{H}_e$ may be infinite-dimensional.
\end{proof}

Now let us focus on extending the second statement of Lemma~\ref{l MO l21}. It is done in the assertion $(ii)$ of the following lemma.

\begin{lemma}\label{l distance two from the sphere to a rank2 element in spin} Let $X$ be a spin factor with $\dim(X)\ge 3$. Assume that $e\in X$ is a unitary tripotent. Then the following assertions are true.
\begin{enumerate}[$(i)$]
	\item  Denote by $S_2(\mathbb{C})$  the space of symmetric $2\times 2$ complex matrices considered as a JB$^*$-subalgebra of $M_2(\mathbb{C})$. Then for each $z\in X$ there is a mapping $\iota:S_2(\mathbb{C})\to X$ with the following properties:
	\begin{enumerate}[$(a)$]
	\item $\iota$ is an {\rm(}isometric{\rm)} unital Jordan $*$-monomorphism of $S_2(\mathbb{C})$ into $X=X_2(e)$ {\rm(}in particular, $\iota(1)=e${\rm)};
	\item $\iota$ is an isometry if $S_2(\mathbb{C})$ is equipped with the normalized Hilbert-Schmidt norm and $X$ is equipped with the hilbertian norm $\|\cdot\|_2$;
	\item $M=\iota(S_2(\mathbb{C}))$ contains $z$ {\rm(}and also $e=\iota(1)${\rm)}.
\end{enumerate}
	\item If, moreover, $z\in S(C)$ and $\|e\pm z\|=2$, then $z\in \mathcal{H}_e\cap M,$ where $M=\iota(S_2(\mathbb{C}))$ is the subtriple given in the previous item for $z$, and, moreover, $z$ is a unitary tripotent in $X$.
\end{enumerate}
\end{lemma}

\begin{proof}
(i) By Lemma~\ref{L:shift in spin}$(ii)$ we can assume that $e\in X^-$. Let $z=a+ib$ with $a,b\in X^-$. Since $\dim(X)\ge3$, we can find an orthonormal system of the form $\{e,c,d\}$ in $X^-$ such that  $a,b\in\mbox{span}_{\mathbb{R}}\{e,c,d\}$. It is now enough to define $\iota$ as the linear extension of the assignment
 $${\begin{pmatrix}
	1&0\\0&1
\end{pmatrix}}\mapsto e,\quad {\begin{pmatrix}
	i&0\\0&-i
\end{pmatrix}}
\mapsto c,\quad {\begin{pmatrix}
	0&i\\i&0
\end{pmatrix}}\mapsto d.$$
Indeed, the three matrices form an orthonormal basis of $S_2(\mathbb{C})$ when equipped with the normalized Hilbert-Schmidt norm, thus $(b)$ is obviously valid. To prove $(a)$ it is enough to observe that $\iota(1)=e$, which is the unit of $X_2(e)$,
$$\begin{gathered}{ \begin{pmatrix}
	i&0\\0&-i
\end{pmatrix}^*=\begin{pmatrix}
	-i&0\\0&i
\end{pmatrix}}\mapsto -c=\J ece=c^{*_e}, \\ {\begin{pmatrix}
	0&i\\i&0
\end{pmatrix}^	*=\begin{pmatrix}
	0&-i\\-i&0
\end{pmatrix}}\mapsto -d=\J ede=d^{*_e}\end{gathered}$$
and
$$ \begin{gathered}
{\begin{pmatrix}
	i&0\\0&-i
\end{pmatrix}^2=\begin{pmatrix}
	-1&0\\0&-1
\end{pmatrix}}\mapsto -e=\J cec=c\circ_e c,\\{ \begin{pmatrix}
	0&i\\i&0
\end{pmatrix}^2=\begin{pmatrix}
	-1&0\\0&-1
\end{pmatrix}}\mapsto -e=\J ded=d\circ_e d,\\  {\begin{pmatrix}
	i&0\\0&-i
\end{pmatrix}\circ\begin{pmatrix}
	0&i\\i&0
\end{pmatrix}}=0\mapsto 0=\J ced=c\circ_e d. \end{gathered}$$
Finally, the validity of $(c)$ is obvious.\smallskip

$(ii)$ Assume $z\in S(C)$ and $\|e\pm z\|=2$. Take the mapping $\iota$ from $(i)$. By property $(a)$ we see that $\iota^{-1}(z)\in S(S_2(\mathbb{C}))$ and $\|1\pm \iota^{-1}(z)\|=2$. Having in mind that $S_2(\mathbb{C})$ is a JB$^*$-subalgebra of $M_2(\mathbb{C})$, we deduce from Lemma~\ref{l MO l21} that $\iota^{-1}(z)\in\mathcal{H}$, so
 $z\in\iota(\mathcal{H}\cap S_2(\mathbb{C}))$. Applying $(a)$ we see that $z\in X_2(e)_{sa}$. Further, since $\mathcal{H}\perp 1$ in the Hilbert-Schmidt inner product, property $(b)$ shows that $z\perp_2 e$ in $X$. Hence $z\in\mathcal{H}_e$. Further, since $\|z\|=\|z\|_2=1$, $z$ is easily seen to be a unitary tripotent in $X$.
 \end{proof}

Next we focus on the way the structure of spin factors may be applied to general rank-$2$ Cartan factors.
The first step is the following  lemma which is essentially contained in the classification of JB$^*$-triples of finite rank (cf. \cite[Theorem 4.10]{Ka81}, \cite{Ka97} and \cite{FriRu85}), and was compiled in \cite[Lemma 2.7]{FerMarPe} from where we have borrowed it.

\begin{lemma}\label{l Peirce-2 of a rank-2 is a spin or CoplusC}\cite[Lemma 2.7]{FerMarPe} Let $e_1$ and $e_2$ be two orthogonal minimal tripotents in a JBW$^*$-triple $M$. Then $M_2(e_1+ e_2)$ is either $\mathbb{C} \oplus^{\infty} \mathbb{C}$ or a spin factor.
\end{lemma}

Let $e_1$ and $e_2$ be two orthogonal minimal tripotents in a JB$^*$-triple $E$. It follows from the weak$^*$-density of $E$ in $E^{**}$ that $e_1$ and $e_2$ are minimal tripotents in $E^{**}$. Clearly, $e_1\perp e_2$ in $E^{**}$. So, by Lemma \ref{l Peirce-2 of a rank-2 is a spin or CoplusC} the Peirce subspace $E^{**}_2(e_1+e_2)$ is either $\mathbb{C} \oplus^{\infty} \mathbb{C}$ or a spin factor. Since $E_2(e_1+e_2)$ is weak$^*$-dense in $E^{**}_2(e_1+e_2),$ and every spin factor is reflexive, we can deduce that $E_2(e_1+e_2)$ also is either $\mathbb{C} \oplus^{\infty} \mathbb{C}$ or a spin factor. So, we get the following improvement of Lemma~\ref{l Peirce-2 of a rank-2 is a spin or CoplusC}.

\begin{lemma}\label{l Peirce-2 of a rank-2 is a spin or CoplusC without duality} Let $e_1$ and $e_2$ be two orthogonal minimal tripotents in a JB$^*$-triple $E$. Then $E_2(e_1+ e_2)$ is either $\mathbb{C} \oplus^{\infty} \mathbb{C}$ or a spin factor.
\end{lemma}

The next ingredient is the following lemma which may be seen as a variant of Lemma~\ref{l distance two from the sphere to a rank2 element in spin}$(i)$ for general Cartan factors.

\begin{lemma}\label{L:quatra}
Let $C$ be a Cartan factor of rank at least $2$ and let $v,w\in C$ be two minimal tripotents. Then either for $B=M_2(\mathbb{C})$ or $B=S_2(\mathbb{C})$ there is an isometric triple monomorphism $\iota:B\to C$ such that $v,w\in\iota(B)$ and
$v=\iota\begin{pmatrix}1&0\\0&0\end{pmatrix}$.
\end{lemma}

\begin{proof}
The assertion follows from
\cite[Lemma 3.10]{FerPe18Adv}). To explain it let us recall some  terminology from \cite{DanFri87,Neher87}.

An ordered quadruple $(u_{1},u_{2},u_{3},u_{4})$ of tripotents in a JB$^*$-triple $E$ is called a \emph{quadrangle} if
$$u_{1}\bot u_{3}, u_{2}\bot u_{4}, u_{1}\top u_{2}\top u_{3}\top u_{4}\top u_{1}\mbox{ and }u_{4}=2 \J
{u_{1}}{u_{2}}{u_{3}}.$$

An ordered triplet $ (v,u,\tilde v)$ of tripotents in $E$, is called a \emph{trangle} if $$v\bot \tilde v, u\vdash v, u\vdash \tilde v\mbox{ and } v = Q(u)\tilde v.$$

Now let us proceed with the proof. By applying Lemma 3.10 in \cite{FerPe18Adv} we conclude that one of the following statements holds:

\begin{enumerate}[$(a)$]\item There exist minimal tripotents $v_2,v_3,v_4$ in $C$
 such that $(v,v_2,v_3,v_4)$ is a quadrangle and $w\in\mbox{span}\{v,v_2,v_3,v_4\}$;
 \item There exist a minimal tripotent $v_2\in V$, and a rank-2 tripotent $u\in C$
 such that $(v, u, v_2)$ is a trangle and $w\in\mbox{span}\{v,u,v_2\}$.
\end{enumerate}

If $(a)$ takes place, we take $B=M_2(\mathbb{C})$ and define $\iota$ by
$$\begin{pmatrix}
	1&0\\0&0
\end{pmatrix}\mapsto v, \begin{pmatrix}
	0&1\\0&0
\end{pmatrix}\mapsto v_2,\begin{pmatrix}
	0&0\\1&0
\end{pmatrix}\mapsto v_4, \begin{pmatrix}
	0&0\\0&1
\end{pmatrix}\mapsto v_3.$$
If $(b)$ takes place, we take $B=S_2(\mathbb{C})$ and define $\iota$ by
$$\begin{pmatrix}
	1&0\\0&0
\end{pmatrix}\mapsto v, \begin{pmatrix}
	0&0\\0&1
\end{pmatrix}\mapsto v_2,\begin{pmatrix}
	0&1\\1&0
\end{pmatrix}\mapsto u.$$
It is easy to check that $\iota$ satisfies the required properties.
\end{proof}

We can now improve the conclusion in Lemma \ref{l distance 2 from a tripotent}.

\begin{lemma}\label{l distance 2 from a tripotent Cartan factor rank geq 2} Let $C$ be a Cartan factor of rank greater than or equal to 2, and let $e$ be a rank-2 tripotent in $C$. Then $C_2(e)$ is a spin factor. Furthermore, suppose $x$ is a norm-one element in $C$ such that $\|e\pm x\| =2.$ Then $P_2(e)(x)$ is a complete tripotent in $C_2(e)$ lying in $\mathcal{H}_e$, and $x= P_2(e)(x) + P_0(e)(x)$. If $C$ has rank-2 then $x$ is a complete tripotent in $C$.
\end{lemma}

\begin{proof} Since $e=v+w$ where $v$ and $w$ are mutually orthogonal tripotents, Lemma \ref{l Peirce-2 of a rank-2 is a spin or CoplusC without duality} shows that $C_2(e)$ is either a spin factor or $\mathbb{C}\oplus^\infty\mathbb{C}$. But the second possibility is excluded by Lemma~\ref{L:quatra}. Indeed, let $B$ and $\iota$ be provided by this lemma. Since $v\perp w$, we deduce that $\iota^{-1}(w)$ is a scalar multiple of
$\begin{pmatrix}0&0\\0&1 \end{pmatrix}$, hence $\iota^{-1}(e)$ is unitary in $B$. It follows that the dimension of $C_2(e)$ is at least $3$.
Hence, $C_2(e)$ must be a spin factor.\smallskip

We consider next the second statement. Under these hypotheses, it follows from Lemma \ref{l distance 2 from a tripotent} that $\|e\pm P_2(e)(x)\| =2$ and $\|P_2(e)(x)\| =1$. Thus, we can deduce from  Lemma~\ref{l distance two from the sphere to a rank2 element in spin}$(ii)$ that $u= P_2(e)(x)$ is a rank-2 (complete) tripotent in $C_2(e)$ lying in $\mathcal{H}_e$. Therefore $P_2(u) (x) = P_2(e) (x) =u$ and \eqref{eq FR 1.6} implies that $$x = P_2(u) (x) + P_0(u)(x) = u + P_0(e) (x).$$

Finally, if we assume that $C$ has rank-2, then $u$ and $e$ must be complete tripotents in $C$, and thus $x= P_2(e) (x) = u$.
\end{proof}

The next proposition on the existence of a real linear extension on a large real linear subspace is a key step in our arguments.

\begin{proposition}\label{p linearity on the hermitian part} Let $\Delta: S(C) \to S(Y)$ be a surjective isometry, where $Y$ is a real Banach space and $C$ is a rank-2 Cartan factor. Let $e$ be a rank-2 tripotent in $C$. Then the restriction $\Delta|_{S(C_2(e)_{sa})}$ admits a real linear extension to $C_2(e)_{sa}$.
\end{proposition}

\begin{proof} Lemma \ref{l distance 2 from a tripotent Cartan factor rank geq 2} proves that $C_2(e)$ is a spin factor whose Hilbert norm is denoted by $\|.\|_2$. By Proposition~\ref{p hermitian part of the Peirce-2 subspace of a rank-2 trip in a spin} we know that
\begin{equation}
\label{eq:l1sum}
C_2(e)_{sa}=\mathbb{R}e\oplus^1 \mathcal{H}_e\end{equation}
where $\mathcal{H}_e$ is a real Hilbert space on which $\|\cdot\|$ coincides with $\|\cdot\|_2$. Corollary \ref{c Tingley antipodal thm for tripotents} gives $\Delta(-e) = - \Delta(e)$ (note that $C$ may be infinite dimensional, thus we cannot apply Tingley's original theorem \cite[Theorem in page 377]{Ting1987}). Further, since on $\mathcal{H}_e$ the two mentioned norms coincide, each element $b\in S(\mathcal{H}_e)$ is a unitary tripotent in $C_2(e)$, in particular another use of Corollary \ref{c Tingley antipodal thm for tripotents} yields $\Delta(-b)=-\Delta(b)$.\smallskip

We will mimic some ideas due to Mori and Ozawa \cite[Lemma 22]{MoriOza2018}. Let $F: C\to Y$ denote the positive homogeneous extension of $\Delta$, that is, $F(0) = 0$ and $F(x) = \|x\| \Delta(\frac{x}{\|x\|})$ for all $x\in C\backslash\{0\}$. To prove that  $\Delta|_{S(C_2(e)_{sa})}$ admits a real linear extension to $C_2(e)_{sa}$ it is enough to show that $F$ is additive on $C_2(e)_{sa}$.\smallskip

The first step is to observe that
\begin{equation}
\label{eq:additivity on edges}
 F(te+sb)=tF(e)+sF(b)\mbox{ whenever }s,t\in\mathbb{R}, b\in S(\mathcal{H}_{e}).\end{equation}
But this is easy if we recall that $\Delta(-e)=-\Delta(e)$, $\Delta(-b)=-\Delta(b)$ and the segments $[(-1)^k e,(-1)^j b]$ are contained in $S(C)$ for $k,j\in \{0,1\}$, thus $\Delta$ is affine on each of this segments by Corollary~\ref{c affine on convex subsets of the sphere in reflexive}.

Let us continue by proving that $F$ is additive on $\mathcal{H}_e$. The first step to this aim is to show that
\begin{equation}\label{eq almost linearity on b,c in He} F(b+c)=F(b)+F(c)\mbox{ whenever }b,c\in \mathcal{H}_e \mbox{ and }\|b\|=\|c\|.
\end{equation}
It it enough to consider the case when $b$ and $c$ are linearly independent (over $\mathbb{R}$) and $\|b\|=\|c\|=1$.

For any $\lambda\in [-1,1]$ we deduce using \eqref{eq:additivity on edges} that
$$\begin{aligned}
\left\| 2(1-|\lambda|) e +\lambda (b+c) \right\| &= \left\| ((1-|\lambda|) e +\lambda b) - (-(1-|\lambda|) e -\lambda c) \right\| \\&= \left\| \Delta((1-|\lambda|) e +\lambda b) - \Delta(-(1-|\lambda|) e -\lambda c) \right\|\\&= \left\| 2(1-|\lambda|) \Delta(e) +\lambda (\Delta(b)+\Delta(c)) \right\|.\end{aligned}$$  Taking $\mu= \|b+c\|= \|\Delta(b) +\Delta(c)\|>0$ and $\lambda = \frac{2}{2+\mu}\in (0,1)$ in the previous identity we deduce that
$$\begin{aligned}
\left\| e \pm \Delta^{-1} \left(\frac{1}{\mu} (\Delta(b) +\Delta(c))\right)\right\|&=\left\|\Delta(e)\pm\frac{1}{\mu} (\Delta(b) +\Delta(c))\right\|
 \\&= \frac{1}{\lambda \mu} \left\| 2(1-\lambda) \Delta(e) \pm \lambda (\Delta(b) +\Delta(c))\right\|
 \\&= \frac{1}{\lambda \mu} \left\| 2(1-|\lambda|) e +\lambda (b+c) \right\| = \left\|e\pm \frac{1}{\mu} (b+c)\right\|= 2,\end{aligned}$$
 where the last equality follows from \eqref{eq:l1sum}. Now, by applying Lemma \ref{l distance 2 from a tripotent Cartan factor rank geq 2} to $e$ and $\displaystyle x=\Delta^{-1} \left(\frac{1}{\mu} (\Delta(b) +\Delta(c))\right) \in S(C)$, we deduce that $x\in S(\mathcal{H}_e)$.

If $\mu\leq 1$ it follows that
$$\begin{aligned}1-\mu + \|\mu x - b\|&= \| (1-\mu) e + \mu x- b\| =  \| \Delta\left((1-\mu) e + \mu x\right)- \Delta(b)\|
\\&=  \|(1-\mu) \Delta\left( e \right)+ \mu \Delta\left( x\right)- \Delta(b)\| =  \|(1-\mu) \Delta\left( e \right)+ \Delta(c)\|\\&= \| F\left(  (1-\mu) e +c\right)\|= \|   (1-\mu) e +c \|=2-\mu,\end{aligned}$$ where we used \eqref{eq:l1sum} (in the first and last equalities) and \eqref{eq:additivity on edges} (in the third and fifth equalities). This proves that $\|\mu x -b\| = 1$.\smallskip

If $\mu \geq 1$, we similarly obtain $$\begin{aligned}1-\frac{1}{\mu} + \left\| x -\frac{1}{\mu} b\right\| &= \left\| x- \left( \left(1-\frac{1}{\mu}\right) e + \frac{1}{\mu} b \right)\right\| =  \left\| \Delta(x) - \Delta\left(\left(1-\frac{1}{\mu}\right) e + \frac{1}{\mu} b\right)\right\|  \\ &= \left\| \Delta(x) - \left(1-\frac{1}{\mu}\right) \Delta\left( e \right) - \frac{1}{\mu} \Delta\left( b\right)\right\| =  \left\|\frac{1}{\mu}\Delta(c) - \left(1-\frac{1}{\mu}\right) \Delta(e)\right\| \\ &=  \left\|\Delta\left( \frac{1}{\mu} c - \left(1-\frac{1}{\mu}\right) e\right) \right\|= 1,
\end{aligned}$$ witnessing that $\left\|x-\frac{1}{\mu} b\right\| = \frac{1}{\mu}$ and thus $\|\mu x -b\| = 1$. We have therefore shown that $\left\| \|b+c\| x - b \right\|= 1$. Similar arguments give $\left\| \|b+c\| x - c \right\|= 1$.\smallskip

Finally, working in the Hilbert space $\mathcal{H}_e$ with the vectors $b,c,x\in S(\mathcal{H}_e)$ satisfying $\left\| \|b+c\| x - b \right\|= 1$ and $\left\| \|b+c\| x - c \right\|= 1,$ we obtain $$\langle x |b \rangle = \langle x | c \rangle = \frac12 \|b+c\|,$$ hence $x=\frac{ b +c}{\|b+c\|}$, which finishes the proof of \eqref{eq almost linearity on b,c in He}. \smallskip

Finally, since $\mathcal{H}_e$ is a Hilbert space, to prove the additivity of $F$ on $\mathcal{H}_e$ it is enough to prove that
\begin{multline*}F(tb_0+sc_0)=tF(b_0)+sF(c_0)\mbox{ whenever }s,t\in\mathbb{R}\\\mbox{ and }b_0,c_0 \mbox{ are orthogonal elements of norm one}.\end{multline*}
Note that the linear span of $b_0$ and $c_0$ is canonically isometric with $\mathbb{C}$ considered as a two-dimensional real Hilbert space. Therefore it is enough to prove the following claim.

\smallskip

{\sc Claim.} Let $Y$ be a real normed space and $G:\mathbb{C}\to Y$ be a continuous positive homogeneous mapping satisfying $G(-z)=-G(z)$ for $z\in\mathbb{C}$ and $G(a+b)=G(a)+G(b)$ if $|a|=|b|$. Then $G$ is real linear.

\smallskip

{\it Proof of the claim.} Since $G$ is positive homogeneous and continuous, it is enough to prove
that
$$G\left(\cos\left(\frac{k \pi}{2^n}\right)  + i\sin\left(\frac{k \pi}{2^n}\right) \right) = \cos\left(\frac{k \pi}{2^n}\right) G( 1) + \sin\left(\frac{k \pi}{2^n}\right)  G(i),$$ for all $n\in \mathbb{N}\cup\{0\}$, $k\in \mathbb{Z}$.
This may be proved by induction on $n$. The case $n=0$ follows from the assumption that $G(-z)=-G(z)$ for $z\in\mathbb{C}$.

Assume the statement holds for some $n\in\mathbb{N}\cup\{0\}$. Take any $k\in\mathbb{Z}$. If $k$ is even, the respective equality (for $n+1$ and $k$) is covered by the induction hypothesis, so it is enough to consider $k$ odd, i.e., $k=2l+1$ for some integer $l$. We set
$\gamma=\left|e^{\frac{il\pi}{2^n}}+e^{\frac{i(l+1)\pi}{2^n}}\right|$ and observe that
$$e^{\frac{i(2l+1)\pi}{2^{n+1}}}=\frac1\gamma(e^{\frac{il\pi}{2^n}}+e^{\frac{i(l+1)\pi}{2^n}}).$$
Hence
$$\begin{aligned}
G\Bigg(\cos\left(\frac{(2l+1) \pi}{2^{n+1}}\right)  +& i\sin\left(\frac{(2l+1) \pi}{2^{n+1}}\right) \Bigg)
=G\left(e^{\frac{i(2l+1)\pi}{2^{n+1}}}\right)\\&=\frac1\gamma\Big(G\left(e^{\frac{il\pi}{2^n}}\right)+G\left(e^{\frac{i(l+1)\pi}{2^n}}\right)\Big)
\\&=\frac1\gamma\Big(\cos\left(\frac{l \pi}{2^n}\right) G(1) + \sin\left(\frac{l \pi}{2^n}\right) G(i) \\&\qquad + \cos\left(\frac{(l+1) \pi}{2^n}\right)G(1) + \sin\left(\frac{(l+1) \pi}{2^n}\right) G(i)\Big)
\\&=\cos\left(\frac{(2l+1) \pi}{2^{n+1}}\right) G(1) + \sin\left(\frac{(2l+1) \pi}{2^{n+1}}\right) G(i),
\end{aligned}$$
which completes the induction step and hence the proof of the claim.\qed
\smallskip

Summarizing, since we have proved that $F$ is real linear on $\mathcal{H}_e$ and \eqref{eq:additivity on edges}, we may conclude that $F$ is real linear on $C_2(e)_{sa}$.
\end{proof}

When in the proof of \cite[Lemma 23]{MoriOza2018}, Proposition \ref{p linearity on the hermitian part} replaces \cite[Lemma 22]{MoriOza2018} we can obtain the next lemma.

\begin{lemma}\label{l 23 MO in triple version for all rank2 Cartan factors} Let $\Delta: S(C) \to S(Y)$ be a surjective isometry, where $Y$ is a real Banach space and $C$ is a rank-2 Cartan factor. Let $\varphi_v$ be a pure atom on $C$, where $v$ is a minimal tripotent in $C$, and let $\psi\in \hbox{supp}_{\Delta}(F_v^{C})$. Suppose $e$ is a rank-2 tripotent such that $v\leq e$. Then $\psi \Delta(x) = \Re\varphi_v (x)$ for all $x\in S(C_2(e)_{sa})$ and for all $x\in \mathbb{T} v \oplus \mathbb{T} (e-v)$.
\end{lemma}

\begin{proof} Set $v_2=e-v$. Then $v_2$ is a minimal tripotent orthogonal to $v$.
 Proposition \ref{p linearity on the hermitian part} assures that $\psi\Delta$ admits a real linear extension, which we will denote by $\varphi$, on $C_2(e)_{sa}= \mathbb{R} e \oplus^{1}\mathcal{H}_e$ (cf. Proposition \ref{p hermitian part of the Peirce-2 subspace of a rank-2 trip in a spin}). In this case $\varphi$ is a norm-one functional in the dual  of $C_2(e)_{sa}$. By the assumption we get that $\varphi(v-v_2) = \psi\Delta(v-v_2) = 1=\varphi_v(v-v_2)$ as $v-v_2\in F_v^C$. Since $v-v_2\in S(\mathcal{H}_e)$ and $\mathcal{H}_e$ is a real Hilbert space, we deduce that
$\varphi(x)=\Re\varphi(x)=\langle x|v-v_2\rangle$ for $x\in\mathcal{H}_e$. Since, moreover, $\varphi (e) = \psi\Delta(e) = 1=\varphi_v(e)$, we infer that $\varphi=\Re\varphi_v$ on  $C_2(e)_{sa}$.\smallskip

For the final assertion fix $\beta\in \mathbb{T}$ and consider the mapping
$$\theta:\alpha\mapsto \psi \Delta (\alpha v + \beta v_2 ),\qquad \alpha\in\mathcal{B}_{\mathbb{C}}.$$
By Corollary \ref{c closed faces associated with tripotents down} we know that this mapping is affine on $\mathcal{B}_{\mathbb{C}}$. Further, $v\pm\beta v_2\in F_v^C$, hence $\psi\Delta(v\pm\beta v_2)=1$. Since $v-\beta v_2$ is a tripotent, Corollary~\ref{c Tingley antipodal thm for tripotents} shows that
$\psi\Delta(-v+\beta v_2)=-1$. It follows that $\theta(0)=\psi\Delta(\beta v_2)=0$, so $\theta$ is (a restriction of) a real linear mapping. Moreover, clearly $\|\theta\|=1$. Taking into account that $\mathbb{C}$ is a real Hilbert space, necessarily
$$\psi\Delta(\alpha v + \beta v_2 )=\theta(\alpha) = \Re \alpha = \Re \varphi_v (\alpha v + \beta v_2 ).$$
\end{proof}

\section{Extending automorphisms and the final step}\label{sec: inner automorphisms}

In this section we shall complete the proof of our main result using Lemma~\ref{l 2.1 FangWang}. The role of the family $(\varphi_i)$ from this lemma will be played by extreme points of the dual unit ball, hence by the functionals $\Re\varphi,$ where $\varphi\in S(E^*)$ is a pure atom (note that
$E$ is considered as a real space). To verify the assumptions of Lemma~\ref{l 2.1 FangWang} we will use the characterization from Corollary~\ref{c additivity on mutually orthogonal minimal tripotents}$(d)$, namely the condition $(d.3)$. This will be done using Proposition~\ref{p linearity on the hermitian part} with the help of some results on automorphisms exchanging minimal tripotents. Let us start by some results on automorphisms and their extensions.\smallskip

Let $a$ be an element in a JB$^*$-triple $E$. It follows from the axioms in the definition of JB$^*$-triples that the operator $L(a,a)$ is hermitian with non-negative spectrum, and hence $e^{i t L(a,a)}$ is a surjective (isometric) automorphism on $E$ for every $t\in \mathbb{R}$. Each automorphism of the form $e^{i t L(a,a)}$ is called an \emph{inner automorphism} on $E$. 
In the case of finite dimensional JB$^*$-triples, inner automorphisms were deeply studied by O. Loos in \cite{Loos2} (see also \cite{Satake80}). 
\smallskip

Let $e$ be a tripotent in $E$. Since $L(e,e) = P_2(e) + \frac12 P_1(e)$, it can be easily checked that \begin{equation}\label{eq inner auto tripotent} e^{i t L(a,a)} = \sum_{k=0}^{\infty} \frac{i^n t^n}{n!} L(e,e)^n = e^{i t} P_2(e) + e^{i \frac{t}{2}} P_1 (e) + P_0(e),
 \end{equation} which coincides with the automorphism $S_{\lambda}$ with $\lambda=e^{i\frac{t}{2}}$ in \cite[Lemma 1.1]{FriRu85}.\smallskip

Suppose $B$ is a JB$^*$-subtriple of a JB$^*$-triple $E$. Clearly, every inner automorphism of the form $e^{i t L(a,a)}$ (where $a\in B$) admits an obvious extension to an inner automorphism of $E$. This can be applied to prove the following two lemmata.

\begin{lemma}\label{l extension by inner automorphism on M2} Let $E$ be a JB$^*$-triple. Suppose $M_2(\mathbb{C})$ is a JB$^*$-subtriple of $E$ and $\Phi: M_2(\mathbb{C})\to M_2(\mathbb{C})$ is the JB$^*$-triple automorphism defined by $$\Phi(x) = \left(
                        \begin{array}{cc}
                          \alpha_1 & 0 \\
                          0 & \alpha_2 \\
                        \end{array}
                      \right) x \left(
                        \begin{array}{cc}
                          \beta_1 & 0 \\
                          0 & \beta_2 \\
                        \end{array}
                      \right),$$ where $\alpha_1,\alpha_2,\beta_1,\beta_2$ are fixed elements in $\mathbb{T}$. Then there  exists a JB$^*$-triple automorphism $\widetilde{\Phi}: E\to E$ whose restriction to $M_2(\mathbb{C})$ is $\Phi$.
\end{lemma}

\begin{proof}
It is clear that $\Phi$ is indeed an automorphism of $M_2(\mathbb{C})$. Further, it is clear that
 $$\Phi(x) =  \gamma \left(
                        \begin{array}{cc}
                          1 & 0 \\
                          0 & \alpha \\
                        \end{array}
                      \right) x \left(
                        \begin{array}{cc}
                          1 & 0 \\
                          0 & \beta \\
                        \end{array}
                      \right) \ (x\in B)$$ where $\gamma=\alpha_1\beta_1$, $\alpha = \alpha_2\alpha_1^{-1}$ and $\beta = \beta_2 \beta_1^{-1}$ are complex units. It is enough to show that the automorphism $\Phi_0=\gamma^{-1}\Phi$ can be extended to $E$.

To this end pick $t_0,t_1\in \mathbb{R}$ such that $e^{i t_0} = \alpha \beta$ and $e^{2i t_1} = \frac{\alpha}{\beta}$. We consider the tripotents
$$e_2= \left(
 \begin{array}{cc}
    0 & 0 \\
     0&1 \\
      \end{array}  \right),\ w_1= \left(
 \begin{array}{cc}
    0 & 1 \\
     0&0  \\
      \end{array}  \right), \mbox{ and }w_2= \left(
 \begin{array}{cc}
    0 & 0 \\
    1 &0 \\
      \end{array}  \right).$$
       We deduce from \eqref{eq inner auto tripotent} that $$ e^{i t_0 L(e_2,e_2)} (x) = \left(
                        \begin{array}{cc}
                          1 & 0 \\
                          0 & e^{i \frac{t_0}{2}} \\
                        \end{array}
                      \right) x \left(
                        \begin{array}{cc}
                          1 & 0 \\
                          0 & e^{i \frac{t_0}{2}} \\
                        \end{array}
                      \right),$$ $$  e^{- i t_1 L(w_1,w_1)} (x) = \left(
                        \begin{array}{cc}
                          0 & e^{- i \frac{t_1}{2}} \\
                          1 & 0 \\
                        \end{array}
                      \right) x \left(
                        \begin{array}{cc}
                          0 & e^{- i \frac{t_1}{2}} \\
                          1 & 0 \\
                        \end{array}
                      \right),$$ and $$  e^{ i t_1 L(w_2,w_2)} (x) = \left(
                        \begin{array}{cc}
                          0 & 1 \\
                           e^{ i \frac{t_1}{2}} & 0 \\
                        \end{array}
                      \right) x \left(
                        \begin{array}{cc}
                          0 & 1  \\
                           e^{i \frac{t_1}{2}} & 0 \\
                        \end{array}
                      \right)$$ for all $x\in B= M_2(\mathbb{C})$.
Since
 $$\Phi_0 = e^{- i t_1 L(w_1,w_1)}\circ e^{ i t_1 L(w_2,w_2)} \circ e^{i t_0 L(e_2,e_2)},$$  we see that $\Phi_0$ is the composition of three inner automorphisms on $M_2(\mathbb{C})$  and hence it can be extended to a JB$^*$-triple automorphism, $\widetilde{\Phi}$, on $E$.
 This completes the proof.
\end{proof}

Similar arguments to those given above are also valid to prove our next lemma. As before, $S_2(\mathbb{C})$ will stand for the Cartan factor of all complex symmetric matrices with complex entries (equivalently, a three dimensional spin factor).

\begin{lemma}\label{l extension by inner automorphism on S2} Let $E$ be a JB$^*$-triple. Suppose $S_2(\mathbb{C})$ is a JB$^*$-subtriple of $E$ and $\Phi: S_2(\mathbb{C})\to S_2(\mathbb{C})$ is the JB$^*$-triple automorphism defined by $$\Phi(x) = \left(
                        \begin{array}{cc}
                          \alpha_1 & 0 \\
                          0 & \alpha_2 \\
                        \end{array}
                      \right) x \left(
                        \begin{array}{cc}
                          \beta_1 & 0 \\
                          0 & \beta_2 \\
                        \end{array}
                      \right),$$ where $\alpha_1,\alpha_2,\beta_1,\beta_2$ are fixed elements in $\mathbb{T}$ with $\alpha_1\beta_2 = \alpha_2 \beta_1$. Then there exists a JB$^*$-triple automorphism $\widetilde{\Phi}: E\to E$ whose restriction to $S_2(\mathbb{C})$ is $\Phi$.
\end{lemma}

\begin{proof} Under the hypothesis of the lemma, we observe that $$\Phi(x) = \left(
                        \begin{array}{cc}
                          \alpha_1 & 0 \\
                          0 & \alpha_2 \\
                        \end{array}
                      \right) x \left(
                        \begin{array}{cc}
                          \beta_1 & 0 \\
                          0 & \beta_2 \\
                        \end{array}
                      \right)=\alpha_1\beta_1 \left(
                        \begin{array}{cc}
                          1 & 0 \\
                          0 & \alpha_2 \overline{\alpha_1} \\
                        \end{array}
                      \right) x \left(
                        \begin{array}{cc}
                          1 & 0 \\
                          0 & \beta_2 \overline{\beta_1} \\
                        \end{array}
                      \right) $$ $$=\alpha_1\beta_1 \left(
                        \begin{array}{cc}
                          1 & 0 \\
                          0 & \alpha \\
                        \end{array}
                      \right) x \left(
                        \begin{array}{cc}
                          1 & 0 \\
                          0 & \alpha \\
                        \end{array}
                      \right) = \alpha_1\beta_1 e^{i t_0 L(e_2,e_2)} (x),$$ where $\alpha = \alpha_2 \overline{\alpha_1} = \beta_2 \overline{\beta_1},$ $e_2= \left(
 \begin{array}{cc}
    0 & 0 \\
     0&1 \\
      \end{array}  \right)$, and $t_0\in \mathbb{R}$ with $e^{i {t_0}} = \beta^2$. This gives the desired conclusion because $e^{i t_0 L(e_2,e_2)}$ is an inner automorphism on $S_2(\mathbb{C})$.
\end{proof}

The previous two lemmata will be used together with Lemma~\ref{L:quatra} and the following one.

\begin{lemma}\label{L:isomorphisms}
Let $B=M_2(\mathbb{C})$ or $B=S_2(\mathbb{C})$ and let $u\in B$ be a unitary element. Set $v=\begin{pmatrix}1&0\\0&0\end{pmatrix}$. Then there are complex numbers $c_1,c_2,d_1,d_2\in \mathbb{T}$
such that the following assertions are valid.
\begin{enumerate}[$(i)$]
	\item The mappings defined by the formula
	$$\Psi_1(x)=\left(\begin{array}{cc}
    {c_1} & 0 \\
    0 & {c_2}
  \end{array}\right) x
  \left(\begin{array}{cc}
    {d_1} & 0 \\
    0 & {d_2}
  \end{array}\right)\mbox{ and }\Psi_2(x)=\left(\begin{array}{cc}
    1 & 0 \\
    0 & {c_2d_1}
  \end{array}\right) x
  \left(\begin{array}{cc}
    1 & 0 \\
    0 & {c_1d_2}
  \end{array}\right)$$
  are automorphisms of $B$;
  \item $\Psi_1^{-1}(u)$ is a hermitian matrix with real entries and zero trace;
  \item $\Psi_1$ and $\Psi_2$ commute with the Peirce projections of $v$;
	\item $\Psi_1=\Psi_2$ on $B_1(v)$;
	\item $\Psi_2(v)=v$.
	\end{enumerate}
\end{lemma}

\begin{proof}
Assume first that $B=M_2(\mathbb{C})$. As observed by Mori and Ozawa in the comments preceding \cite[Lemma 23]{MoriOza2018}, in this case, there exist complex numbers $c_1,c_2,d_1,d_2\in \mathbb{T}$ and $t\in [0,1]$ such that
\begin{equation}
\label{eq:shift in M2}
u = \left(\begin{array}{cc}
    \overline{c_1} & 0 \\
    0 & \overline{c_2}
  \end{array}\right) \left(                                                     \begin{array}{cc}
                                                                          t & \sqrt{1-t^2} \\
                                                                          \sqrt{1-t^2} & -t
                                                                        \end{array}
                                                                     \right)
  \left(\begin{array}{cc}
    \overline{d_1} & 0 \\
    0 & \overline{d_2}
  \end{array}\right).\end{equation}
The numbers chosen in this way obviously have all the properties.

Next assume that $B=S_2(\mathbb{C})$. The remark from \cite{MoriOza2018} may be applied as well,  so we may choose
 $c_1,c_2,d_1,d_2\in \mathbb{T}$ such that \eqref{eq:shift in M2} holds. Since the middle matrix on the right-hand side is symmetric, the only
 additional thing to be assured is that the mappings $\Psi_1$ and $\Psi_2$ preserve symmetry of matrices, i.e. that $c_1d_2=c_2d_1$. If $t\ne 0$, it is satisfied automatically due to the symmetry of $u$. If $t=0$,
 it is not satisfied automatically, but $u$ is then a diagonal matrix and hence we may easily achieve even $c_1=d_1$ and $c_2=d_2$.
 This completes the proof.
\end{proof}

The next proposition is the last step to the proof of our main result.

\begin{proposition}\label{p behavior on pure atoms rank2 Cartan factor} Let $\Delta: S(C) \to S(Y)$ be a surjective isometry, where $Y$ is a real Banach space and $C$ is a rank-2 Cartan factor. Let $\varphi_v$ be a pure atom in $C^*$, where $v$ is a minimal tripotent in $C$, and let $\psi\in \hbox{supp}_{\Delta}(F_v^{C})$. Then $\psi \Delta(x) = \Re\varphi_v (x),$ for all $x\in S(C)$.
\end{proposition}

\begin{proof} We observe that Corollary \ref{c additivity on mutually orthogonal minimal tripotents}$(d)$ tells that it suffices to prove that $$\hbox{$\psi \Delta(w) = \varphi_v (w)$ for every minimal tripotent $w\in C$.}$$

We fix an arbitrary minimal tripotent $w\in C$. Let $B$ and $\iota$ be given by Lemma~\ref{L:quatra} applied to $v$ and $w$.
Set $e= \iota(1)$. Another use of Corollary \ref{c additivity on mutually orthogonal minimal tripotents}$(d)$ shows that to prove
the equality $\psi\Delta(w)=\Re\varphi_v(w)$ it is enough to
prove that $$\psi \Delta(u) = \varphi_v (u), \hbox{ for every rank-$2$ tripotent $u\in \iota(B)$.}$$
Let us fix a rank-$2$ tripotent $u\in \iota(B)$. Then $\iota^{-1}(u)$ is a rank-$2$ tripotent in $B$, hence it is a unitary element of $B$.
Let $\Psi_1$ and $\Psi_2$ be the automorphisms of $B$ provided by Lemma~\ref{L:isomorphisms}. For $j=1,2$ let $\Phi_j$ be an automorphims of $C$ extending $\iota\circ\Psi_j\circ\iota^{-1}$. It exists by Lemma~\ref{l extension by inner automorphism on M2} or by Lemma~\ref{l extension by inner automorphism on S2}.

By property $(ii)$ from Lemma~\ref{L:isomorphisms} we have
$$\iota^{-1}(\Phi_1^{-1}(u))=\Psi_1^{-1}(\iota^{-1}(u))=\begin{pmatrix}
	t&s\\s&-t
\end{pmatrix}$$
for some $s,t\in\mathbb{R}$. Hence we have
$$\begin{aligned}
\psi\Delta(u)&=\psi\Delta\Phi_1\left(\iota\begin{pmatrix}
	t&s\\s&-t
\end{pmatrix}\right)\\&=
t\psi\Delta\Phi_1\left(\iota\begin{pmatrix}
	1&0\\0&-1
\end{pmatrix}\right)+s\psi\Delta\Phi_1\left(\iota\begin{pmatrix}
	0&1\\1&0
\end{pmatrix}\right),
\end{aligned}$$
by Proposition~\ref{p linearity on the hermitian part} applied to the surjective isometry $\Delta \Phi_1|_{S(C)}$.\smallskip

We consider next the surjective isometry $\Delta \Phi_2|_{S(C)}$. Since $\Phi_2^{-1}(v)=v$  the respective pure atom is
$$\varphi_v\Phi_2=\varphi_{\Phi_2^{-1}(v)}=\varphi_{v}$$
and the associate face is
$$\Phi_2^{-1}(F_v^C)= F_{\Phi_2^{-1}(v)}^C=F_{v}^C.$$
Since
$$\psi\in \mbox{supp}_{\Delta} (F_v^C)=\mbox{supp}_{\Delta\Phi_2}(\Phi_2^{-1}(F_v^C))=\mbox{supp}_{\Delta\Phi_2}(F_{v}^C),$$
we get
$$\begin{aligned}\psi\Delta\Phi_1\left(\iota\begin{pmatrix}
	0&1\\1&0
\end{pmatrix}\right)&=\psi\Delta\Phi_2\left(\iota\begin{pmatrix}
	0&1\\1&0
\end{pmatrix}\right)=\Re\varphi_v\left(\iota\begin{pmatrix}
	0&1\\1&0
\end{pmatrix}\right)=0\\&=\Re\varphi_v\left(\Phi_1\iota\begin{pmatrix}
	0&1\\1&0
\end{pmatrix}\right).\end{aligned}$$
Indeed, the first equality follows from property $(iv)$ in Lemma~\ref{L:isomorphisms}, the second one from
Lemma \ref{l 23 MO in triple version for all rank2 Cartan factors} and the last two equalities follow from the definition of $\varphi_v$
together with property $(iii)$ from Lemma~\ref{L:isomorphisms}.\smallskip

Another application of property $(iii)$ from Lemma~\ref{L:isomorphisms} gives $$\Phi_1\iota\begin{pmatrix}
	1&0\\0&-1
\end{pmatrix}=\iota\Psi_1\begin{pmatrix}
	1&0\\0&-1
\end{pmatrix}\in\mathbb{T}v\oplus\mathbb{T}(e-v),$$ hence by Lemma \ref{l 23 MO in triple version for all rank2 Cartan factors}
we get
$$\psi\Delta\left(\Phi_1\iota\begin{pmatrix}
	1&0\\0&-1
\end{pmatrix}\right)=\Re\varphi_v\left(\Phi_1\iota\begin{pmatrix}
	1&0\\0&-1
\end{pmatrix}\right).$$
Combining the above we obtain
$$\begin{aligned}
\psi\Delta(u)&=
t\psi\Delta\Phi_1\left(\iota\begin{pmatrix}
	1&0\\0&-1
\end{pmatrix}\right)+s\psi\Delta\Phi_1\left(\iota\begin{pmatrix}
	0&1\\1&0
\end{pmatrix}\right)
\\&=t\Re\varphi_v\left(\Phi_1\iota\begin{pmatrix}
	1&0\\0&-1
\end{pmatrix}\right)+s\Re\varphi_v\left(\Phi_1\iota\begin{pmatrix}
	0&1\\1&0
\end{pmatrix}\right)
\\&=\Re\varphi_v\Phi_1\iota\begin{pmatrix}
	t&s\\s&-t
\end{pmatrix}=\Re\varphi_v(u)
\end{aligned}$$
which completes the proof.
\end{proof}

\medskip\medskip

\textbf{Acknowledgements}
O.F.K.Kalenda was partially supported by the Czech Science Foundation, project no. GA\v{C}R 17-00941S.

A.M. Peralta partially supported by the Spanish Ministry of Science, Innovation and Universities (MICINN) and European Regional Development Fund project no. PGC2018-093332-B-I00 and Junta de Andaluc\'{\i}a grant FQM375. \smallskip

The research of this article was partially done during a visit of A.M. Peralta to Charles University in Prague. He would like to thank for the hospitality during his stay.

\end{document}